\documentclass[11pt,twoside]{article}
\usepackage{etex}
\usepackage[dvipsnames]{xcolor} % 颜色宏包调用

\usepackage{color}
\usepackage{booktabs}
\usepackage{amsthm}

\usepackage{amsfonts,amssymb,amsmath,amsxtra,url,float} % 论文初始格式设定 ams 宏包
\allowdisplaybreaks[4] % 允许多行公式强制换页
\usepackage[colorlinks,
  linkcolor=magenta, %
  anchorcolor=Periwinkle,
  citecolor=violet,
  urlcolor=blue
  ]{hyperref} % 超链颜色设定
\usepackage{enumitem}
\usepackage{geometry} 
\usepackage{rotating} % 允许rotating环境旋转段落
\usepackage{lscape} % 允许使用landscape环境进行局部旋转
\usepackage{yhmath} % 让括号变得更圆润, 但对某些版本的tex会使得大括号发生断裂
\usepackage{multirow}
\usepackage{graphicx} % 图片宏包调用, 调用此宏包后将被允许插入png 格式图片
\usepackage{subfigure} % figure环境下的补充宏包, 允许在figure环境下对自图片进行分栏
\usepackage{tikz}
\usepackage{pgfplots}
\usepackage{tikz-3dplot}
\usetikzlibrary{patterns}
\usetikzlibrary{3d,calc}
\usetikzlibrary{decorations.pathreplacing,decorations.markings}
 \tikzset{
  on each segment/.style={
    decorate,
    decoration={
      show path construction,
      moveto code={},
      lineto code={
        \path [#1]
        (\tikzinputsegmentfirst) -- (\tikzinputsegmentlast);
      },
      curveto code={
        \path [#1] (\tikzinputsegmentfirst)
        .. controls
        (\tikzinputsegmentsupporta) and (\tikzinputsegmentsupportb)
        ..
        (\tikzinputsegmentlast);
      },
      closepath code={
        \path [#1]
        (\tikzinputsegmentfirst) -- (\tikzinputsegmentlast);
      },
    },
  },
  mid arrow/.style={postaction={decorate,decoration={
        markings,
        mark=at position 0.6 with {\arrow[#1]{stealth}} % 可根据具体情况调整数字0.6
      }}},
}
\usetikzlibrary{arrows}
\usetikzlibrary{trees}

\usetikzlibrary{matrix}
\usetikzlibrary{patterns}
\usetikzlibrary{shadings} % 调用tikz作图指令包
\usepackage{fancyhdr} % 设置页眉、页脚 % see \title[]{...}
\def\headertitle{A Counterexample to Left-Right Symmetry for Diagonal Auslander Regularity}
\def\fstpage{1} % 起始页码设置
\def\page{$\begin{matrix} {\color{white}0} \\ \thepage \end{matrix}$} % 页脚页码代码转换

\usepackage[all]{xy} % 交换图箭图宏包调用, 允许在xymatrix环境下进行交换图表编译
\usepackage{dsfont} % 可以调用mathds 指令
\usepackage{cite}
\usepackage{mathrsfs} % mathscr
\numberwithin{figure}{section}
\usepackage{marginnote} % 可使用边注\marginnote{}指令
\usepackage{graphicx} % 可使用\figure 环境
\usepackage{multicol} %用于实现在同一页中实现不同的分栏

\usepackage{enumitem}
\setenumerate[1]{itemsep=0pt,partopsep=0pt,parsep=\parskip,topsep=3pt}
\setitemize[1]{itemsep=0pt,partopsep=0pt,parsep=\parskip,topsep=3pt}
\setdescription{itemsep=0pt,partopsep=0pt,parsep=\parskip,topsep=3pt}
\setlist[itemize]{leftmargin=35pt}
\setlist[enumerate]{leftmargin=35pt}
\usepackage{changepage} % changepage宏包可以支持段落的双侧缩进
\newtheorem{theorem}{Theorem}[section]
\newtheorem{lemma}[theorem]{Lemma}

\newtheorem{main theorem}[theorem]{Main Theorem}

\newtheorem{question}[theorem]{Question}

\usetikzlibrary{arrows}

\numberwithin{equation}{section}

\def\orcid{
\begin{tikzpicture}[baseline=-1mm]
\filldraw[Green!35] (0,0) circle (5pt);
\filldraw[white] (0,0) node{\tiny\textbf{iD}};
\end{tikzpicture}
}
\def\orcid{
\begin{tikzpicture}[baseline=-1mm]
\filldraw[Green!35] (0,0) circle (5pt);
\filldraw[white] (0,0) node{\tiny\textbf{iD}};
\end{tikzpicture}
}

\newcommand{\ORCIDNOTATION}[1]{\href{https://orcid.org/#1}{\orcid}}
\def\EnglishTitle{A Counterexample to Left-Right Symmetry for Diagonal Auslander-Regularity}
\def\PaperAuthorsENname{
Zhenyu Wang
$^{\ref{Author1}}$,
Zhaoyong Huang
$^{\ref{Author1}}$,
Yu-Zhe Liu
$^{\ref{Author2},\ref{Author3}}$ %\ORCIDNOTATION{\SecondAuthorORICD}\ref{orcid3}}$
}
\def\FirstEnOrgani{School of Mathematics, Nanjing University, Nanjing 210093, Jiangsu Province, P.R. China}
\def\SecondEnOrgani{School of Mathematics and Statistics, Guizhou University,\\
Guiyang 550025, Guizhou Province, P.R. China}
\def\ThirdEnOrgani{Guizhou Provincial Key Laboratory of Applied Mathematics and Computing Power \& Algorithms, \\
Guiyang 550025, Guizhou Province, P.R. China}

\def\kk{\Bbbk} % mathds双写字体的"k", 用于描述域
\def\Q{\mathcal{Q}} % mathcal花体的"Q", 用于描述箭图
\def\I{\mathcal{I}}

\def\compos{\ \lower-0.2ex\hbox{\tikz\draw (0pt, 0pt) circle (.1em);} \ }

\newcommand{\Hom}{\mathrm{Hom}} %
\newcommand{\Ext}{\mathrm{Ext}} %
\newcommand{\pd}{\mathrm{pd}}
\newcommand{\fd}{\mathrm{fd}}
\newcommand{\id}{\mathrm{id}}

\newcommand{\op}{\mathrm{op}}
\newcommand{\defines}[1]{{\it\color{violet}#1}}

\DeclareMathOperator{\gldim}{gl.dim}
\DeclareMathOperator{\grade}{grade}

\newcommand{\myBmod}[1]{%
  \begingroup
  \renewcommand{\arraystretch}{0.8}%
  \setlength{\arraycolsep}{1.5pt}%
  \expandafter\myBmodSplitA#1,\myBmodEnd
}
\def\myBmodSplitA#1,#2,#3,#4,#5,#6\myBmodEnd{%
  \def\myBmodOne{#1}\def\myBmodTwo{#2}\def\myBmodThree{#3}%
  \def\myBmodFour{#4}\def\myBmodFive{#5}%
  \expandafter\myBmodSplitB#6\myBmodEnd
}
\def\myBmodSplitB#1,#2,#3,#4,#5,\myBmodEnd{%
  \scriptstyle
  \begin{array}{@{}cccccc@{}}
    &&&& #5 & \\
    && \myBmodThree & #1 & #2 & \\
    \myBmodOne & \myBmodTwo & \myBmodFive & #3 && #4 \\
    & \myBmodFour &&&&
  \end{array}%
  \endgroup
}

\title{\bf \EnglishTitle
}

\vspace{5mm}

\author{\PaperAuthorsENname}
\date{ }

\begin{document}

\thispagestyle{empty}
\maketitle

%\begin{enumerate}[leftmargin=6.6cm]\footnotesize
%  \item[\orcid]
%    \ORCID{\SecondAuthorORICD}
%    \label{orcid3}
%\end{enumerate}

\vspace{2mm}
\begin{enumerate}[label=\textbf{\color{red}\arabic*}]\footnotesize
  \item
    \begin{center}
      \FirstEnOrgani
    \end{center}
    \label{Author1}
  \item
    \begin{center}
      \SecondEnOrgani
    \end{center}
    \label{Author2}
  \item
    \begin{center}
      \ThirdEnOrgani
    \end{center}
    \label{Author3}
\end{enumerate}

\vspace{1mm}
\begin{adjustwidth}{1cm}{1cm}
  \noindent\footnotesize
\textbf{Abstract}: We construct a finite-dimensional Auslander-regular algebra $R$
with global dimension three, which is right diagonal but not left diagonal; moreover,
any simple right $R$-module with grade $i$ has projective dimension $i$,
but there is a simple left $R$-module with grade one and projective dimension two.
It answers a question raised by Iyama negatively.

  \vspace{1mm}

  \vspace{1mm}
\end{adjustwidth}

%\newpage
%\tableofcontents

\section{Introduction}

For a ring $R$ and a right $R$-module $M$, we use $$0 \to M \to E^0(M) \to E^1(M)
\to \cdots \to E^i(M) \to \cdots$$ to denote a minimal injective
coresolution of $M$, and use $\pd_RM$, $\fd_RM$ and $\id_RM$ to denote the projective,
flat and injective dimensions of $M$, respectively. The number $\grade M:=\inf\{i\geqslant 0\mid\Ext^i_R(M,R)\neq 0\}$
is the \defines{grade} of $M$. We use $\gldim R$ to denote the global dimension of $R$.

In \cite{B} Bass proved a fundamental theorem which states that a commutative noetherian ring $R$
is a Gorenstein ring (that is, $\id_RR<\infty$)
if and only if $\fd_RE^i(R)\leqslant i$ for any $i\geqslant 0$.
Let $R$ be a left and right noetherian ring. Auslander proved that the above condition is left-right symmetric \cite[Theorem 3.7]{FGR};
in this case, $R$ is said to satisfy the {\it Auslander condition}. In addition, the ring $R$ is called {\it Auslander $n$-Gorenstein}
with $n\geqslant 1$ if $\fd_RE^i(R_R)\leqslant i$ for any $0\leqslant i\leqslant n-1$. Thus $R$ satisfies the Auslander condition
if and only if it is Auslander $n$-Gorenstein for all $n$.
A left and right noetherian ring $R$ is called {\it Auslander-regular}
if it satisfies the Auslander condition and $\gldim R<\infty$ \cite{C}. Since the Auslander condition is left-right symmetric,
so is the notion of Auslander-regular rings. In \cite{I} Iyama called an Auslander-regular ring
$R$ {\it right diagonal} if $\fd_RI=i$ for any non-zero direct summand $I$ of $E^i(R_R)$ and $i\geqslant 0$. Symmetrically,
the notion of
{\it left diagonal Auslander-regular rings} is defined.
He proved that the notion of diagonal Auslander-regular artin algebras is left-right symmetric in the case for the global dimension
of the algebra being at most two \cite[Theorem 3.4]{I}, and that for an Auslander 1-Gorenstein artin algebra $R$,
the condition that a simple $R$-module $S$ with $\grade S=0$ is projective is left-right symmetric \cite[3.4.1]{I}.
Then he asks whether these two results hold true in general.

\begin{question}[{\!\cite[Question 3.6.2]{I}}] \rm  \label{quest:Iyama}
Let $R$ be an artin algebra. Is the condition that $R$ is diagonal Auslander-regular left-right symmetric?
More generally, is the following condition $(*)$ left-right symmetric for an Auslander $n$-Gorenstein ring $R$?
\begin{itemize}
\item [$(*)$] A simple $R$-module $S$ with $\grade S=i$ satisfies $\pd_RS=i$ for any $0\leqslant i\leqslant n-1$.
\end{itemize}
\end{question}

In this paper, we construct a finite-dimensional Auslander-regular algebra $R$
with $\gldim R=3$, which is right diagonal but not left diagonal; moreover,
any simple right $R$-module $S$ with $\grade S=i$ satisfies $\pd_RS=i$,
but there is a simple left $R$-module $S'$ with $\grade S'=1$ and $\pd_{R^{\op}}S'=2$.
This answers Question \ref{quest:Iyama} negatively.

\section{A counterexample}

Let $\kk$ be an algebraically closed field and let $R:=\kk \Q/\I$ be a finite-dimensional $\kk$-algebra given by the quiver $\Q$:
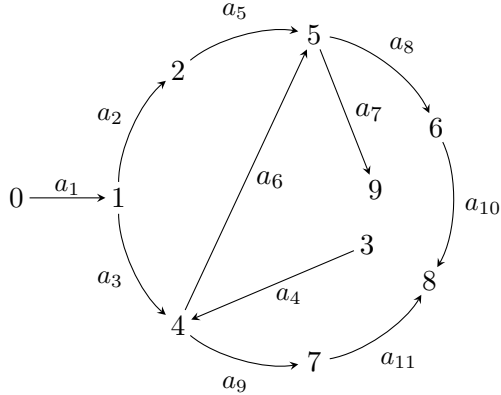
\begin{figure}[htbp]
\centering
\begin{tikzpicture}[
  >=stealth,
  vertex/.style={inner sep=2pt},
  edge label/.style={fill=white, inner sep=1pt, font=\small},
  circle arrow/.style={->, shorten <=6pt, shorten >=6pt}
]
  % Vertices on the circle, in clockwise order.
  \node[vertex] (1) at (180:2.2cm) {$1$};
  \node[vertex] (2) at (130:2.2cm) {$2$};
  \node[vertex] (5) at (80:2.2cm) {$5$};
  \node[vertex] (6) at (25:2.2cm) {$6$};
  \node[vertex] (8) at (-30:2.2cm) {$8$};
  \node[vertex] (7) at (-80:2.2cm) {$7$};
  \node[vertex] (4) at (-130:2.2cm) {$4$};

  % Three outward branches.
  \node[vertex] (0) at (180:3.55cm) {$0$};
  \node[vertex] (3) at (-30:1.25cm) {$3$};
  \node[vertex] (9) at (  5:1.20cm) {$9$};

  \draw[->] (0) -- node[edge label, above] {$a_1$} (1);
  \draw[circle arrow] (1.center) arc[start angle=180,end angle=130,radius=2.2cm];
  \draw[circle arrow] (1.center) arc[start angle=180,end angle=230,radius=2.2cm];
  \draw[->] (3) -- node[edge label, below right] {$a_4$} (4);
  \draw[circle arrow] (2.center) arc[start angle=130,end angle=80,radius=2.2cm];
  \draw[->] (4) -- node[right] {$a_6$} (5);
  \draw[->] (5) -- node[right] {$a_7$} (9);
  \draw[circle arrow] (5.center) arc[start angle=80,end angle=25,radius=2.2cm];
  \draw[circle arrow] (4.center) arc[start angle=-130,end angle=-80,radius=2.2cm];
  \draw[circle arrow] (6.center) arc[start angle=25,end angle=-30,radius=2.2cm];
  \draw[circle arrow] (7.center) arc[start angle=-80,end angle=-30,radius=2.2cm];

  % Place the circular-arrow labels outside the circle.
  \node[edge label] at (155:2.55cm) {$a_2$};
  \node[edge label] at (205:2.55cm) {$a_3$};
  \node[edge label] at (105:2.55cm) {$a_5$};
  \node[edge label] at (52.5:2.55cm) {$a_8$};
  \node[edge label] at (-105:2.55cm) {$a_9$};
  \node[edge label] at (-2.5:2.62cm) {$a_{10}$};
  \node[edge label] at (-55:2.62cm) {$a_{11}$};
\end{tikzpicture}
\caption{The quiver $\Q$ of $R$.}% The relations are given in \eqref{eq:relations}.}
\label{fig:quiver}
\end{figure}

\noindent
with the ideal $\I$ given by the relations
$$a_2a_5=a_3a_6, a_6a_8a_{10}=a_9a_{11}\ \text{and}\ a_5a_7=a_4a_9=a_4a_6a_8=0.$$
For any $i\in\Q_0 = \{0,1,\dots,9\}$, associated with the vertex $i$, we write
$$P(i):=\text{the indecomposable projective right}\ R\text{-module},$$
$$I(i):=\text{the indecomposable injective right}\ R\text{-module},$$
$$S(i):=\text{the indecomposable simple right}\ R\text{-module},$$
$$P'(i):=\text{the indecomposable projective left}\ R\text{-module},$$
$$I'(i):=\text{the indecomposable injective left}\ R\text{-module},$$
$$S'(i):=\text{the indecomposable simple left}\ R\text{-module}.$$

\begin{lemma} \label{lemm:pdI idP right}
It holds that\\
{\rm (1)}
\begin{align}
 & (\pd_RI(i))_{i=0}^9 =(1, 2, 2, 3, 2, 1, 1, 1, 0, 0), \label{lemm:pdI idP right 1} \\
 & (\id_RP(i))_{i=0}^9 =(0, 1, 1, 0, 1, 2, 2, 1, 3, 2), \label{lemm:pdI idP right 2}\\
 & (\pd_RS(i))_{i=0}^9=(1, 2, 2, 3, 2, 1, 1, 1, 0, 0), \label{lemm:simpmod-pd and id 1} \\
 & (\id_RS(i))_{i=0}^9=(0, 1, 1, 0, 1, 2, 2, 2, 3, 2). \label{lemm:simpmod-pd and id 2}
\end{align}
{\rm (2)}
\begin{align}
 & (\pd_{R^{\op}}I'(i))_{i=0}^9=(0, 1, 1, 0, 1, 2, 2, 1, 3, 2), \label{lemm:pdI idP left 1} \\
 & (\id_{R^{\op}}P'(i))_{i=0}^9=(1, 2, 2, 3, 2, 1, 1, 1, 0, 0), \label{lemm:pdI idP left 2}\\
 & (\pd_{R^{\op}}S'(i))_{i=0}^9=(0, 1, 1, 0, 1, 2, 2, 2, 3, 2), \label{lemm:simpmod-pd and id 3} \\
 & (\id_{R^{\op}}S'(i))_{i=0}^9=(1, 2, 2, 3, 2, 1, 1, 1, 0, 0). \label{lemm:simpmod-pd and id 4}
\end{align}
{\rm (3)} $\gldim R=3$.\\
{\rm (4)} The minimal injective coresolutions of $R_R$ and $_RR$ are as follows.
\begin{align*}
&0\rightarrow R_R \rightarrow I(8)^{\oplus 8}\oplus I(9)^{\oplus 4}\rightarrow I(0)\oplus I(5)^{\oplus 4}\oplus I(6)^{\oplus 2}\oplus I(7)^{\oplus 4}\\
&\rightarrow I(1)\oplus I(2)\oplus I(4)^{\oplus 2}\rightarrow I(3)\rightarrow 0, \\
& and\\
& 0\rightarrow {_RR}\rightarrow I'(0)^{\oplus 8}\oplus I'(3)^{\oplus 4}\rightarrow I'(1)\oplus I'(2)^{\oplus 3}\oplus I'(4)^{\oplus 5}\oplus I'(7)\\
&\rightarrow I'(5)\oplus I'(6)^{\oplus 2}\oplus I'(7)\oplus I'(9)\rightarrow I'(8)\rightarrow 0.
\end{align*}
\end{lemma}

\begin{proof}
(1) It is straightforward to verify these equations by constructing the minimal projective resolution
or injective coresolution of each corresponding module. For example,
for the indecomposable injective right $R$-module $I(0)$, its minimal projective resolution is
$$0\rightarrow P(1)\rightarrow P(0)\rightarrow I(0)\rightarrow0,$$
and so $\pd_RI(0)=1$.

(2) According to (1), we get (2) immediately by using the usual duality $\Hom_{\kk}(-,\kk)$.

Assertion (3) follows directly from (1) or (2). Assertion (4) follows by a direct calculation.
\end{proof}

\begin{lemma}\label{prop:grade}
It holds that
\begin{align}
 & (\grade S(i))_{i=0}^9=(1,2,2,3,2,1,1,1,0,0), \label{prop:grade 1} \\
 & (\grade S'(i))_{i=0}^9=(0,1,1,0,1,2,2,1,3,2). \label{prop:grade 2}
\end{align}
\end{lemma}

\begin{proof}
By \cite[Lemma 1]{Iw}, we have
$$\Ext_R^t(S(i),R) \cong \Hom_R(S(i), E^t(R_R))$$
for any $0\leqslant i\leqslant 9$ and $t\geqslant 0$. Thus, by Lemma \ref{lemm:pdI idP right}(4), for $S(0)$, we have
\begin{align*}
  \Hom_R(S(0),R) \cong
  \Hom_R(S(0), E^0(R_R)) \cong \Hom_R(S(0), I(8)^{\oplus 8} \oplus I(9)^{\oplus 4}),
\end{align*}
and
\begin{align*}
  \Ext_R^1(S(0),R)  \cong \Hom_R(S(0), E^1(R_R))
 \cong \Hom_R(S(0), I(0)\oplus I(5)^{\oplus 4} \oplus I(6)^{\oplus 2} \oplus I(7)^{\oplus 4}).
\end{align*}
Since $\Hom_R(S(0), I(8))=0$ and $\Hom_R(S(0), I(9))=0$, we have $\Hom_R(S(0),R)=0$.
Since $\Hom_R(S(0), I(0)) \ne 0$, we have $\Ext_R^1(S(0),R)\ne 0$.
Thus, $\grade S(0)=1$. Similarly, we may compute the grades of other simple right $R$-modules and simple left $R$-modules,
and obtain \eqref{prop:grade 1} and \eqref{prop:grade 2}.
\end{proof}

By Lemma \ref{lemm:pdI idP right}, we have that $R$ is Auslander-regular with $\gldim R=3$.
Moreover, for any $0\leqslant i\leqslant 3$ and any non-zero direct summand $I$ of $E^i(R_R)$, one can check that $\pd_RI=i$.
It follows that $R$ is right diagonal Auslander-regular.
On the other hand, by Lemma \ref{lemm:pdI idP right}(4)(2), the indecomposable injective left $R$-module $I'(7)$
occurs as a direct summand of $E^2(_RR)$, but $\pd_{R^{\op}} I'(7)=1$. It means that $R$ is not left diagonal Auslander-regular.
Consequently, diagonal Auslander-regularity is not left-right symmetric.

By Lemmas \ref{lemm:pdI idP right} and \ref{prop:grade},
we have $\eqref{lemm:simpmod-pd and id 1} = \eqref{prop:grade 1}$, that is,
$$\pd_RS(i) = \grade S(i)\ \text{for any}\ 0\leqslant i\leqslant 9.$$
By Lemmas \ref{lemm:pdI idP right} and \ref{prop:grade} again, we have
$$\eqref{lemm:simpmod-pd and id 3} = (0,1,1,0,1,2,2,2,3,2) \ne (0,1,1,0,1,2,2,1,3,2) = \eqref{prop:grade 2},$$
To be precise, we have
$$\pd_{R^{\op}}S'(7)=2\neq 1=\grade S'(7).$$
Therefore, the algebra $R$ shows that condition $(*)$ is not left-right symmetric when $n\geqslant 2$.

\vspace{0.5cm}

\section*{Acknowledgements}
This work was supported by the NSF of China (Grant No. 12371038, 12401042 and 12561008).


\begin{thebibliography}{99}
\bibitem{B} H. Bass, {\it On the ubiquity of Gorenstein rings}, Math. Z. {\bf 82} (1963), 8--28.
\bibitem{C} J. Clark, {\it Auslander--Gorenstein rings for beginners}, International Symposium on Ring
Theory (Kyongju, 1999),  Trends Math., Birkhauser Boston, Boston, MA, 2001, pp.95--115.
\bibitem{I} O. Iyama, {\it The relationship between homological properties and representation theoretic realization of artin algebras},
Trans. Amer. Math. Soc. {\bf 357}(2) (2005), 709--734.
\bibitem{FGR} R.M. Fossum, P.A. Griffith and I. Reiten, Trivial Extensions of Abelian Categories.
Homological Algebra of Trivial Extensions of Abelian Categories with Applications to Ring Theory,
Lecture Notes in Math. {\bf 456}, Springer-Verlag, Berlin-Heidelberg-New York, 1975.

%    \bibitem{HZ} C. H. Huang, L. J. Zheng, {\it Simple modules over Auslander regular rings}, Open Mathematics, {\bf 15} (2017), 1618--1622.
%
%    \bibitem{Lam} T. Y. Lam,  {\it Lectures on Modules and Rings}, Graduate Texts in Mathematics, Volume 189, Springer, 1998.
\bibitem{Iw} Y. Iwanaga, {\it On rings with finite self-injective dimension}, Communications in Algebra, {\bf 7}(4) (1979), 393--414.

\end{thebibliography}
\end{document}